\documentclass[twoside]{formrev}

\usepackage{amssymb}
\usepackage{amsmath}
\usepackage{amsthm}
\usepackage{latexsym}
\usepackage{eucal}

\usepackage[utf8]{inputenc}   
\usepackage{xspace}
\usepackage{graphicx}
\graphicspath{{Fig/}}
\usepackage{subfigure}

\usepackage{algorithmic}

\title{Algebraic invariants of graphs; a study based on computer exploration}

\author{Nicolas M. Thiéry}
\affiliation{
  Laboratoire de Mathématiques Discrètes\\
  Université Lyon I, 43 bd du 11 novembre\\
  69622 Villeurbanne Cedex\\
  \texttt{nthiery@users.sourceforge.net}\\
  \texttt{http://www.mines.edu/\string~nthiery/}}

\newtheorem{theorem}{Theorem}[section]
\newtheorem{lemma}[theorem]{Lemma}
\newtheorem{proposition}[theorem]{Proposition}
\newtheorem{corollary}[theorem]{Corollary}

\newtheorem{problem}[theorem]{Problem}

\newtheorem{conjecture}[theorem]{Conjecture}
\newtheorem{condition}[theorem]{Condition}
\newtheorem{caracterisation}[theorem]{Caracterisation}

\newtheoremstyle{algo}
  {3pt}
  {3pt}
  {}
  {}
  {\bfseries}
  {}
  {.5em}
  {}

{\theoremstyle{algo} \newtheorem{algorithm}[theorem]{Algorithm}}

\newcommand{\ie}{\emph{i.e.~}}
\newcommand{\apriori}{\emph{a priori}\xspace}

\renewcommand{\ij}{{\{i,j\}}}		
\newcommand{\sym}{\mathfrak S}
\newcommand{\N}{\mathbb{N}}
\newcommand{\K}{\mathbb{K}}
\newcommand{\C}{\mathbb{C}}

\renewcommand{\c}{\mathbf{c}}
\newcommand{\x}{\mathbf{x}}
\newcommand{\e}{\mathbf{e}}
\newcommand{\E}{\mathbf{E}}
\newcommand{\g}{\mathbf{g}}
\newcommand{\h}{\mathbf{h}}
\newcommand{\G}[1][n]{\mathcal{G}_{#1}}
\newcommand{\V}[1][n]{\mathcal{V}_{#1}}
\renewcommand{\v}{\mathbf{v}}
\newcommand{\w}{\mathbf{w}}
\newcommand{\Pol}[1][n]{\K[x_\ij]}
\newcommand{\Inv}[1][n]{\mathcal{I}_{#1}} 
\newcommand{\Inf}{\Inv[\infty]}
\newcommand{\OInv}{\overrightarrow\Inv}

\newcommand{\permuvar}{\texttt{PerMuVAR}\xspace}
\newcommand{\mupad}{\texttt{MuPAD}\xspace}
\newcommand{\magma}{\texttt{Magma}\xspace}
\newcommand{\maple}{\texttt{Maple}\xspace}
\newcommand{\invar}{\texttt{Invar}\xspace}
\newcommand{\fgb}{\texttt{FGb}\xspace}

\DeclareMathOperator{\MGSGS}{s}
\DeclareMathOperator{\GL}{GL}
\DeclareMathOperator{\SL}{SL}
\DeclareMathOperator{\Aut}{Aut}
\DeclareMathOperator{\sign}{sign}

\newcommand{\myvcentermiddle}[1]{%
  \setbox0\hbox{#1}%
  \dimen255=.5\dp0%
  \advance\dimen255 by -.5\ht0%
  \advance\dimen255 by .5ex%
  \ifvmode\hskip0em\fi\raise\dimen255\box0}

\newcommand{\st}{\,\,|\,\,}

\newcommand{\oast}{{\odot\!\!\!\!\ast}}
\newcommand{\exps}[1]{{\x^{{#1}}}^{\oast}}
\newcommand{\expm}[1]{\x^{#1}}
\newcommand{\gnuplotscale}{0.66}

\newcommand{\graph}[2][scale=\graphscale]{\includegraphics[#1]{#2}}
\renewcommand{\graph}[2][scale=\graphscale]{\myvcentermiddle{\begingroup
    \includegraphics[#1]{#2} \endgroup}}

\newcommand{\pgraph}[2][scale=\graphscale]{{\left(\!\graph[#1]{#2}\!\!\!\right)^\oast}\,}

\begin{document}

\maketitle

\toappear{Not printed}
\pagestyle{myheadings}
\def\runninghead{Thi\'ery: Algebraic invariants of graphs}

\begin{abstract}
  We consider the ring $\Inv$ of polynomial invariants over weighted
  graphs on $n$ vertices. Our primary interest is the use of this ring
  to define and explore algebraic versions of isomorphism problems of
  graphs, such as Ulam's reconstruction conjecture.

  There is a huge body of literature on invariant theory which
  provides both general results and algorithms. However, there is a
  combinatorial explosion in the computations involved and, to our
  knowledge, the ring $\Inv$ has only been completely described for
  $n\le4$.

  This led us to study the ring $\Inv$ in its own right. We used
  intensive computer exploration for small $n$, and developed
  \permuvar, a library for \mupad, for computing in invariant rings of
  permutation groups.

  We present general properties of the ring $\Inv$, as well as results
  obtained by computer exploration for small $n$, including the
  construction of a medium sized generating set for $\Inv[5]$. We
  address several conjectures suggested by those results (low degree
  system of parameters, unimodality), for $\Inv$ as well as for more
  general invariant rings.  We also show that some particular sets are
  not generating, disproving a conjecture of Pouzet related to
  reconstruction, as well as a lemma of Grigoriev on the invariant
  ring over digraphs.  We finally provide a very simple minimal
  generating set of the field of invariants.
\end{abstract}

\section*{Introduction}

Let $\K$ be a field of characteristic zero, $n$ be a positive integer,
and $\{x_{\{1,2\}}, \dots, x_{\{n-1,n\}}\}$ be a set of $\binom n 2$
variables indexed by the pairs $\ij$ of $\{1,\dots,n\}$. The symmetric
group $\sym_n$ acts naturally on those variables by
\begin{displaymath}
  \sigma\cdot x_\ij:=x_{\{\sigma(i),\sigma(j)\}}.
\end{displaymath}
Let $\Pol$ be the ring of polynomials in $x_\ij$. We study the subring
$\Inv:=\Pol^{\sym_n}$ of the polynomials which remain invariant under
the action of $\sym_n$.

Our motivation comes from graph theory and in particular from graph
reconstruction. Pouzet~\cite{Pouzet.1977,Pouzet.1979} formulated an
algebraic reconstruction conjecture for $\Inv$, which implies Ulam's
famous reconstruction conjecture for weighted
graphs~\cite{Bondy.1991}. We disprove Pouzet's conjecture. Kocay
proposed a similar conjecture, by introducing the \emph{algebra of
  subgraphs}~\cite{Kocay.1982,Cameron.1996}. This algebra is a
quotient of $\Inv$; however this quotient is not graded, and we cannot
apply our method to disprove Kocay's conjecture.

The ring $\Inv$ can also be used to study the shape of sets of
vectors~\cite{Aslaksen_al.1996}. Our primary goal is to construct
complete systems of invariants (systems that separate weighted graphs
up to isomorphism), and in particular minimal generating sets of
$\Inv$.

In \S~\ref{preliminaries}, we introduce the representation $\G$ of the
symmetric group $\sym_n$ over the vector space $\V$ of weighted graphs
on $n$ vertices, and the associated invariant ring $\Inv$. We review
classical results and tools provided by invariant theory (finite
generation, grading, Hilbert series). Since $\G$ is a permutation
group, there is a combinatorial interpretation of the invariant ring,
and a reasonably fast algorithm for computing the Hilbert series. We
also review some general properties of minimal generating sets, and
the definition of the smallest degree bound $\beta(\Inv)$.

\S~\ref{generating_sets} is devoted to generating sets of $\Inv$. We
provide a finite generating set.  By studying the Hilbert series, we
show that two other sets are not generating, disproving Pouzet's
conjecture. We also prove that, for many common monomial orders,
$\Inv$ has no finite SAGBI basis.

Finding a good degree bound is crucial. In \S~\ref{Hironaka}, we
recall how Hironaka decompositions of $\Inv$ can be used to obtain the
degree bound $\beta(\Inv)\le\binom{\binom n 2}2-\mu_n$, where $\mu_n$
is a non-negative $O(n)$ integer. We calculate $\mu_n$ by constructing
minimal multigraphs without odd automorphisms.

In \S~\ref{parameters}, we try to refine the degree bound by
constructing low degree systems of parameters. The study of the
Hilbert series for $n\le21$, combined with a conjecture of Mallows
and Sloane, suggests the existence of a system of parameters composed
of invariants of degrees $1,2,\dots,n, 2,3,\dots, \binom{n-1}2$. This
would give \mbox{$\beta(\Inv)\le\binom{\binom{n-1}2}2-\mu_n$}. We
propose a natural construction for such a low degree system of
parameters, and check its validity for $n\le5$ using a Gröbner basis
computation. Unfortunately, this computation is intractable for
$n\ge6$. Such a system of parameters seems to have nearly optimally
low degrees; therefore, this technique cannot be refined much further
in order to get better degree bounds.

\S~\ref{MGS} is devoted to the computation of minimal generating sets.
For $n=4$, a minimal generating set was first constructed by hand by
Aslaksen, Chan et Gulliksen~\cite{Aslaksen_al.1996}; it can now be
computed in a few seconds by invariant theory software (e.g. Kemper's
packages in \maple~\cite{Kemper.Invar} or \magma~\cite{Kemper.1998}).
However, for $n\ge5$, these software packages are unable to compute
even partial minimal generating sets. We wrote
\permuvar~\cite{Thiery.P.2001}, a library of invariant theory routines
for \mupad, which uses the usual
algorithms~\cite{Sturmfels.AIT,Kemper.1998}, but is specialized for
permutation groups. This allows us to go a step further: for $n=5$, we
compute a partial minimal generating set, containing $57$ polynomials
of degree $\le 9$\footnote{Using \emph{ad hoc} computations,
  Kemper~\cite{Kemper.IG5} checked recently that this system was
  indeed a complete minimal generating set, thus proving that
  $\beta(\Inv[5])=9$.}. This suggests a much better degree bound:
$\beta(\Inv)=\binom n 2 -1$.

In \S~\ref{Gorenstein}, we prove that the invariant ring $\Inv$ is
Gorenstein when $n$ is even. This fact could be used to accelerate the
computations of Hironaka decompositions~\cite{Thiery.P.2001}.

We introduce in \S~\ref{chain_product} the chain product (a naive
interpretation of Stanley-Reisner rings~\cite{Garsia_Stanton.1984}).
This allows for faster computations of generating sets at the expense
of non-minimality~\cite{Thiery.P.2001}: we obtain a generating set of
$\Inv[5]$ containing about one thousand polynomials of degree $\le22$.

In \S~\ref{infinity} the projective limit $\Inf$ is used to obtain
results about $\Inv$; this includes the lower bound $\beta(\Inv)\ge
\lfloor\frac n 2\rfloor$.

\S~\ref{unimodality} presents various unimodality properties revealed
by computer exploration, for $\Inv$ as well as for more general
invariant rings.

Grigoriev~\cite{Grigoriev.1979} introduces a related invariant ring
over digraphs. In \S~\ref{digraphs}, we apply the Hilbert series tool
of \S~\ref{generating_sets} to disprove lemma~1 of
\cite{Grigoriev.1979}. We also provide a simple counter-example.
Finally, in \S~\ref{invariant_field}, we study the field of
invariants. Grigoriev~\cite{Grigoriev.1979} gives a non-constructive
proof for the existence of a small generating set of the field of
invariants (the proof of the degree bound is incorrect though, since
it relies on lemma~1 of \cite{Grigoriev.1979}). We construct such a
small generating set, composed of the elementary symmetric polynomials
and a very simple invariant of degree $2$; to the contrary of
Grigoriev's assertion, it is not a complete system of invariants. We
also derive a minimality property of the invariant ring, by using
basic Galois theory on the field of invariants.

The results presented in this paper are part of the Ph.~D.
thesis~\cite{Thiery.IAGR} of the author. We refer to this document for
the detailed proofs.

\section{The invariant ring over graphs}

\label{preliminaries}

\subsection{Valuated graphs as a vector space}

\label{valuated_graphs}

Let $V$ be a $\K$-vector space of finite dimension $m$, and $G$ be a
finite subgroup of $\GL(V)$. Tacitly, we interpret $G$ as a group of
$m\times m$ matrices or as a representation on $V$. Two vectors $\v$
and $\w$ are \emph{isomorphic}, or in the same \emph{$G$-orbit} (for
short \emph{orbit}), if $\sigma\cdot\v=\w$ for some $\sigma\in G$.

Let $n$ be a positive integer. We consider labelled, undirected graphs
on the vertices $\{1,\dots,n\}$, without loops, and whose edges are
weighted in $\K$. A \emph{simple graph} is a graph with weights in
$\{0,1\}$, and a \emph{multigraph} is a graph with weights in $\N$.
For any pair $\ij$, let $\e_\ij$ be the simple graph with one single
edge $\ij$. The set of all graphs is a $\K$-vector space $\V$ of
dimension $m:=\binom n 2$ with basis
$\{\e_{\{1,2\}},\dots,\e_{\{n-1,n\}}\}$. Indeed, any graph $\g$ can be
written uniquely as $\g:=\sum g_\ij \e_\ij$, where $g_\ij$ is the
weight of the edge $\ij$. Let $\{x_{\{1,2\}},\dots,x_{\{n-1,n\}}\}$ be
the dual basis ($x_\ij(\g)$ is the weight $g_\ij$ of the graph $\g$ on
the edge $\ij$).

Throughout the text, we denote objects attached to $\V$ by cursive
symbols, and objects attached to the generic vector space $V$ by
ordinary symbols. Let $\sym_n$ be the symmetric group of all
permutations of the $n$ vertices. Our group $\G$ is the linear
representation of $\sym_n$ defined on the basis of $\V$ by
$\sigma\cdot\e_\ij:=\e_{\{\sigma(i),\sigma(j)\}}$. The notion of
isomorphism defined above coincides with the usual notion of
isomorphism of graphs. Orbits of labelled graphs are called
\emph{unlabelled graphs}. Unless otherwise stated, all graphs are
unlabelled.

The representation of $\G$ on $\V$ splits into three irreducible
components: $[n]\oplus[n-1,1]\oplus[n-2,2]$, where $[n-2,2]$
represents the irreducible representation of $\sym_n$ indexed by the
partition $\lambda=(n-2,2)$ of $n$
\cite{Aslaksen_al.1996,Fulton_Harris.RT}. The first component has
dimension $1$ and corresponds to the vector space spanned by the
complete graph. The sum $[n]\oplus[n-1,1]$ of the first two components
is of dimension $n$, and corresponds to the vector space spanned by
the $n$ \emph{stars} $\E_1,\dots,\E_n$, where $\E_i:=\sum_{j\neq i}
\e_\ij$.
This representation is the natural representation of $\sym_n$ by
permutation of $\E_1,\dots,\E_n$. Let $X_1,\dots,X_n$ be the basis of
the dual, defined by $X_i:=\sum_{j\neq i} x_\ij$. If $\g$ is a graph,
$X_i(\g)$ is the degree of the vertex $i$ of $\g$. Finally, the last
irreducible component $[n-2,2]$ is the orthogonal of the two previous
components, that is the subspace of all \emph{$0$-regular graphs}
(graphs where each vertex as degree $0$).

\subsection{The invariant ring}

Recall that if $G$ acts on $V$, a \emph{complete system of invariants}
is a set $S$ of functions such that two elements $\v$ and $\w$ of $V$
are in the same orbit if and only if they give the same value to all
functions in $S$ (\ie $p(\v)=p(\w)$ for all $p\in S$). Our primary
goal is to construct, or at least find information about, complete
systems of invariants. We introduce the invariant ring of $G$ which
provides a mechanical way to do this. We refer
to~\cite{Stanley.1979,Sturmfels.AIT,Cox_al.IVA,Smith.1997,Kemper.1998}
for classical literature on invariant theory of finite groups. Parts
of what follows are strongly inspired by~\cite{Kemper.1998}.

Let $(x_1,\dots,x_m)$ be a basis of the dual of $V$; for $\V$, we take
$(x_1,\dots,x_m):=(x_{\{1,2\}},\dots,x_{\{n-1,n\}})$. Let
$\K[x_1,\dots,x_m]$ be the ring of polynomials over $V$. The action of
$G$ on $V$ extends naturally to an action of $G$ on $K[x_1,\dots,x_m]$
by $\sigma\cdot p:=p\circ \sigma^{-1}$. An \emph{invariant polynomial}, or
\emph{invariant}, is a polynomial $p\in K[x_1,\dots,x_m]$ such that
$\sigma\cdot p=p$ for all $\sigma\in G$. The \emph{invariant ring} $I(G)$ is
the set of all invariants. We call $\Inv:=I(\G)$ the \emph{invariant
  ring over graphs}. Note that
$\sigma\cdot x_\ij:=x_\ij\circ\sigma^{-1}=x_{\{\sigma(i),\sigma(j)\}}$.

Obviously, $I(G)$ is a $\K$-algebra. Hilbert's famous theorem states
that $I(G)$ is \emph{finitely generated}: there exists a finite set
$S$ of invariants such that any invariant can be expressed as a
polynomial combination of invariants in $S$. We call $S$ a
\emph{generating set}. If no proper subset of $S$ is generating, $S$
is a \emph{minimal generating set}. Since $I(G)$ is finitely
generated, there exists a degree bound $d$ such that $I(G)$ is
generated by the set of all invariants of degree at most $d$. We
denote by $\beta(I(G))$ the \emph{smallest degree bound}.

There exist algorithms to compute (minimal) generating sets, and a
basic result of invariant theory states that they are complete systems
of invariants. However, this often leads to very intensive
computations, and rather large complete systems of invariants.

\subsection{Invariant ring of a permutation group}

\label{invariants_permutation}

The most famous invariant ring is the ring of symmetric polynomials
$I(\sym_m)$, defined by the natural action of $\sym_m$ on the
variables $(x_1,\dots,x_m)$. The fundamental theorem of symmetric
polynomials~\cite[p.~2]{Sturmfels.AIT} states that $I(\sym_m)$ is
generated by $m$ algebraically independent symmetric polynomials, for
example the $m$ elementary symmetric polynomials or the first $m$
symmetric power sums.

$\G$ is a \emph{permutation group}, since it acts by permuting the
variables $x_\ij$; thus, we also view $\G$ as a subgroup of the full
group $\sym_m$ of the permutations of $m$ variables. This results in
several convenient and powerful combinatorial interpretations of
invariants:

(i) A labelled multigraph $\g:=(g_{\{1,2\}},\dots,g_{\{n-1,n\}})$ can
be identified with the monomial
$\expm\g:=x_{\{1,2\}}^{g_{\{1,2\}}}\dots
x_{\{n-1,n\}}^{g_{\{n-1,n\}}}$. The \emph{exponential} of $\g$ is the
polynomial $\exps\g:=\sum_\h\expm\h$, where $\h$ belongs to the orbit
of $\g$. The polynomial $\exps\g$ is invariant, and is well defined
even if $\g$ is unlabelled. The exponential therefore identifies
unlabelled graphs with some particular invariants. Moreover, the set
of all invariants $\exps\g$, where $\g$ is a multigraph, is a vector
space basis of $\Inv$. Note that the exponential differs from the
usual Reynolds operator ${}^*$ by a multiplicative factor:
$\exps\g=|\Aut(\g)|({\expm\g})^*$, where $|\Aut(\g)|$ is the size of
the automorphism group of $\g$.

(ii) Let $\g_1$ and $\g_2$ be two multigraphs on $n$ vertices. The
product $\exps{\g_1}\exps{\g_2}$ is a linear combination of all
possible superpositions of $\g_1$ and $\g_2$ (with, at times,
counter-intuitive coefficients). For instance:
\begin{displaymath}
  \pgraph{autogen/ugraph-1110000000}\times\pgraph{autogen/ugraph-1000000000}=
  \pgraph{autogen/ugraph-2110000000}
  +\pgraph{autogen/ugraph-1111000000}
  +\pgraph{autogen/ugraph-1110000001}.
\end{displaymath}
Let $n'>n$, and consider the multigraphs $\overline\g_1$ and
$\overline\g_2$ obtained by adding $n'-n$ isolated vertices to $\g_1$
and $\g_2$. New superpositions, which fit in $n'$ vertices but not in
$n$ vertices, may appear in the product
$\exps{\overline\g_1}\exps{\overline\g_2}$. However, the use of a
modified Reynolds operator ensures that the coefficients in the linear
combination do not change. This makes the product somewhat independent
of $n$ (see \S~\ref{infinity}).

(iii) If $\g$ and $\h$ are simple graphs, $\exps\g(\h)$ counts the
number $s(\g,\h)$ of subgraphs of $\h$ isomorphic to $\g$. The
following invariants can be used to count respectively the number of
edges, the number of pairs of adjacent edges and the number of
Hamiltonian cycles of $\h$:
\begin{xalignat*}{3}
  \pgraph{autogen/ugraph-1000000000}(\h), &&
  \pgraph{autogen/ugraph-1100000000}(\h), &&
  \pgraph{autogen/ugraph-1010011001}(\h).
\end{xalignat*}
Manipulations of the quantities $s(\g,\h)$ are the cornerstone of
several results on reconstruction of graphs \cite{Bondy.1991};
for the use of these algebraic considerations
see~\cite{Pouzet_Thiery.IAGR.2001}.

\subsection{Grading, Hilbert series and degree bound}

Powerful properties of an invariant ring are its grading and the
associated Hilbert series. As a $\K$-vector space, $I(G)$ is not
finite dimensional. However, since the action of $G$ preserves the
degree of polynomials, $I(G)$ decomposes into the direct sum of its
homogeneous components:
\begin{displaymath}
  I(G)=\bigoplus_{d=0}^\infty I(G)_d,
\end{displaymath}
where $I(G)_d$ is the finite dimensional vector space of all
homogeneous invariants of degree $d$. The Hilbert series of $I(G)$ is
the generating series of its dimensions:
\begin{displaymath}
  H(I(G),z):=\sum_{d=0}^\infty z^d \dim I(G)_d.
\end{displaymath}

For general finite groups of matrices, this series can be computed by
averaging over the group, through Molien's formula. However, since
$\G$ is a permutation group, the set of all invariants $\exps\g$,
where $\g$ is a multigraph with $d$ edges is a vector space basis of
$\Inv[n,d]$. Therefore, computing $H(\Inv,z)$ reduces to a Pólya
enumeration of multigraphs with respect to the number of
edges~\cite{Stanley.1979}. Recall that the conjugacy classes
$C_\lambda$ of $\sym_n$ are indexed by the partitions $\lambda$ of
$n$. Let $\sigma$ be a permutation of the vertices in $C_\lambda$. The
cycle type of the induced permutation of the edges is easily computed
from the cycle type of $\sigma$, \ie from the partition
$\lambda$~\cite{Harary_Palmer.GE}. We denote by
$l_1(\lambda),\dots,l_m(\lambda)$ this cycle type. Then,
\begin{displaymath}
  H(\Inv,z)=\frac{1}{n!}\sum_{\lambda} |C_\lambda|\frac 1{\prod(1-z^i)^{l_i(\lambda)}},
\end{displaymath}
where the sum is over all partitions $\lambda$ of $n$. This provides
an algorithm whose complexity is about $O(n^4\exp(n^{0.8}))$. Concretely,
we can compute $H(\Inv,z)$ for $n\le21$. It is sometimes useful to
consider the multigraded Hilbert series, where each grading
corresponds to one of the three irreducible components of the
representation $\G$. We can compute this multigraded Hilbert series
for $n\le15$.

Given an integer $d\ge1$, let $\K[I(G)_{<d}]$ be the subalgebra of
$I(G)$ generated by invariants of degree $<d$, and $\K[I(G)_{<d}]_d$
its homogeneous component of degree $d$. Set $\MGSGS_0(I(G)):=0$ and
$\MGSGS_d(I(G)):=\dim I(G)_d - \dim \K[I(G)_{<d}]_d$. The generating
series $\MGSGS(I(G),z):=\sum_{d=0}^\infty z^d \MGSGS_d(I(G))$ is a
polynomial of degree $\beta(I(G))$.

A set $S$ is \emph{homogeneous} if its elements are also homogeneous.
The following lemma (valid for any graded algebra $A$, where $A_0$ is
the ground field $\K$) summarizes some general properties of
generating sets.
\begin{lemma} Let $S$ be a generating set of $I(G)$.

  (i) $I(G)$ has a homogeneous minimal generating set composed of at
  most $|S|\beta(I(G))$ invariants of degree at most $\beta(I(G))$.

  (ii) Assume $S$ is homogeneous, and let $S_d$ be the set of all
  invariants of $S$ having degree $d$. Then, $S$ is a minimal
  homogeneous generating set if and only if for all $d$, $S_d$ is a
  vector space basis of a direct factor of $\K[I(G)_{<d}]_d$ in
  $I(G)_d$. In particular, $|S_d|=\MGSGS_d(I(G))$.
\end{lemma}
\begin{proof}
  (i) For each $p\in S$ and $d$, let $p_d$ be the homogeneous
  component of degree $d$ of $p$. Since $I(G)$ is graded, it is
  generated by the set $\{p_d \st p\in S, 1\le d\le\beta(I(G))\}$.

  (ii) Use the grading and basic linear algebra.
\end{proof}
From (i), it is not very restrictive to consider only homogeneous
generating sets, since non-homogeneous generating sets are not much
smaller than homogeneous ones.

The Hilbert series provides a simple necessary condition to test if a
set $S$ of homogeneous invariants is generating. The following
proposition is valid for any graded algebra $A$, where $A_0$ is the
ground field $\K$. We stress the importance of the homogeneity of the
invariants. A series $s(z)$ is \emph{dominated} by a series $t(z)$ if
the coefficients of $s(z)$ are upper-bounded term by term by the
coefficients of $t(z)$.
\begin{condition}
  \label{cond.dimension}
  Let $S:=(p_1,\dots,p_t)$ be a homogeneous generating set, with
  respective degrees $(d_1,\dots,d_t)$. Then, the Hilbert series
  $H(I(G),z)$ is dominated by the series
  \begin{displaymath}
    F(d_1,\dots,d_t,z):=\frac{1}{(1-z^{d_1})\dots(1-z^{d_t})}.
  \end{displaymath}
\end{condition}
\begin{proof}
  As a vector space, the homogeneous component $I(G)_d$ is generated
  by the set of all the homogeneous products $p_1^{\lambda_1}\dots
  p_t^{\lambda_t}$ whose degree $d_1\lambda_1+\dots+d_t\lambda_t$ is
  $d$; those products are counted by the series $F(d_1,\dots,d_t,z)$.
\end{proof}
This apparently weak condition is in fact very powerful. In
particular, it leads to the proof of theorem~\ref{th.not_gen}, and to
the disproving of Grigoriev's lemma~1 of \cite{Grigoriev.1979} (see
\S~\ref{digraphs}).

\section{Generating sets of $\Inv$}

\label{generating_sets}

For $n=1,2,3$, the invariant ring is the ring of symmetric
polynomials; the elementary symmetric polynomials form a minimal
generating set. For $n=4$, Aslaksen et al.~\cite{Aslaksen_al.1996}
constructed by hand the following minimal generating set:
\begin{gather*}
  \left\{
    \pgraph{autogen/ugraph-100000}, \pgraph{autogen/ugraph-110000},
    \pgraph{autogen/ugraph-001100},
    \pgraph{autogen/ugraph-111000}, \pgraph{autogen/ugraph-110100},
  \right.\\
  \left.
    \pgraph{autogen/ugraph-101100},
    \pgraph{autogen/ugraph-111100},
    \pgraph{autogen/ugraph-101101}, \pgraph{autogen/ugraph-111110}
  \right\}.
\end{gather*}
At about the same time, we had proven independently a similar result
through a Gröbner basis computation with CoCoA~\cite{CoCoA}, using
theorem~2.7.9 of~\cite{Sturmfels.AIT}. However, our set was not
minimal since we had not removed the invariant $\exps\c$ where $\c$ is
the complete graph. The set above can now be computed in about one
second with Kemper's implementation of $\invar$ in
$\magma$~\cite{Kemper.1998}.

We now provide a large, but reasonable, finite generating set of
$\Inv$. A multigraph is \emph{quasi-connected} if it
has, at most, one non-trivial connected component. For example, $\g_1$
below is quasi-connected, but not $\g_2$:
\begin{xalignat*}{2}
  \g_1:=\graph{autogen/ugraph-0002001000}, &&
  \g_2:=\graph{autogen/ugraph-0012001000}.
\end{xalignat*}
\begin{proposition}
  \label{prop.connected_generate}
  (i) The homogeneous component $\Inv[n,d]$ has for vector space basis
  the set of all invariants $\exps{\c_1}\cdots\exps{\c_k}$, where each
  $\c_i$ is a quasi-connected multigraph with $n_i$ non-isolated
  vertices and $d_i$ edges, and where \mbox{$n_1+\dots+n_k\le n$} and
  $d_1+\dots+d_k=d$.

  (ii) The invariant ring $\Inv$ is generated by the set of all
  homogeneous invariants $\exps\g$, where $\g$ is a quasi-connected
  multigraph with at most $\beta(\Inv)$ edges.
\end{proposition}
\begin{proof}
  (i) Let $\g$ be a multigraph with $n$ vertices and $k>1$ non-trivial
  connected components $\c_1,\dots,\c_k$. Let
  $\overline\c_1,\dots,\overline\c_1$ be the quasi-connected
  multigraphs on $n$ vertices obtained by adding isolated vertices to
  the $\c_i$. Obviously, $n_1+\dots+n_k\le n$, and $d_1+\dots+d_k=d$.
  Then:
  \begin{displaymath}
    \exps\g=
    \exps{\overline\c_1}\cdots\exps{\overline\c_k}-\sum_i\exps{\h_i},
  \end{displaymath}
  where the $\h_i$ are multigraphs with strictly less than $k$
  non-trivial connected components. For example:
  \begin{displaymath}
    \pgraph{autogen/ugraph-1110000001}=
    \pgraph{autogen/ugraph-1110000000}\times\pgraph{autogen/ugraph-0000000001}
    -\pgraph{autogen/ugraph-1111000000}
    -\pgraph{autogen/ugraph-2110000000}.
  \end{displaymath}
  By induction on the number $k$ of non-trivial connected components,
  $\g$ is a linear combination of products
  $\exps{\overline\c_1}\cdots\exps{\overline\c_k}$. In fact, we just
  inverted a triangular linear system with ones on the diagonal, and
  uniqueness follows.

  (ii) Use (i) and the definition of $\beta(\Inv)$.
\end{proof}
Obviously, in order to get a usable generating set, it is essential to
have a good bound for $\beta(\Inv)$.

We tried to use the technique of SAGBI basis. This is a powerful tool,
which generalizes Gröbner basis techniques for rings instead of ideals
\cite{Robbiano_Sweedler.1990}. The main drawback is that there exist
invariant rings with no finite SAGBI basis; this seems to be the case
for $\Inv$, at least for many common monomial orders.
\begin{theorem}
  There are no finite SAGBI basis for $\Inv$ if the monomial order is
  either lexicographic, degree lexicographic, or degree reverse
  lexicographic with the $n-1$ smallest variables corresponding to
  adjacent edges.
\end{theorem}
\begin{proof}
  We prove in each case that there is an infinite number of
  irreducible initial monomials (an initial monomial is irreducible if
  it cannot be written as product of two smaller initial monomials).
  For the lexicographic order, we can alternatively use Göbel's
  characterization of permutation groups with finite SAGBI
  basis~\cite{Goebel.1998}.
\end{proof}

The following theorem states that some sets are not generating. (i)
disproves a tempting generalization of the fundamental theorem of
symmetric functions, whereas (ii) disproves Pouzet's
conjecture~\cite{Pouzet.1977}, which would have implied Ulam's
reconstruction conjecture.
\begin{theorem}
  \label{th.not_gen}
  (i) For $n\ge5$, the set of all invariants $\exps\g$, where $\g$ is
  a simple graph, do not generate $\Inv$.

  (ii) For $11\le n\le18$, the set of all invariants $\exps\g$, where
  $\g$ is a multigraph with at least one isolated vertex, do not
  generate $\Inv$.
\end{theorem}
\begin{proof}
  (i) For $n=5,6,7,8$, simple graphs can be counted with respect to
  the number of edges using Pólya enumeration~\cite{Harary_Palmer.GE}.
  The coefficient of degree $d=4$ of the series $S(d_1,\dots, d_t)$ is
  strictly smaller than that of the Hilbert series. Therefore,
  condition~\ref{cond.dimension} applies. For $n\ge9$, no new
  isomorphism types of multigraphs with less than $4$ edges appears,
  so the coefficient of degree $4$ of both series is the same as for
  $n=8$. Condition~\ref{cond.dimension} again applies.

  (ii) By an argument similar to the proof of
  proposition~\ref{prop.connected_generate}, we have only to consider
  the set of all invariants $\exps\g$, where $\g$ is a multigraph with
  a unique non-trivial connected component, which is of size $<n$.
  Those multigraphs can be counted from the total number of
  multigraphs by using a technique similar to that described
  in~\cite[\S~4.2, p.~90]{Harary_Palmer.GE}. For $11\le n\le18$,
  computations of both series shows that
  condition~\ref{cond.dimension} fails.
\end{proof}
We could not check (ii) for $n>18$ since the computation were
intractable. However, the results for $n\le 18$ strongly suggest that,
for $d\approx 4n-24$, the ratio between the coefficients of degree $d$
of the two series is bounded by an expression of the form
$\exp(-an)+0.17$. This could probably be confirmed by an asymptotic
study, and we conjecture that (ii) is true for any $n\ge11$.

\section{Decomposition of Hironaka}

\label{Hironaka}

The smallest degree bound $\beta(\Inv)$ and furthermore the polynomial
$\MGSGS(\Inv,z)$ contains important information about the invariant
ring, which would be very useful when the computation of a minimal
generating set is intractable. But so far, we don't know how to
calculate them except by explicitly computing such a minimal
generating set.

Invariant theory provides only some bounds on $\beta(\Inv)$ and
$\MGSGS(\Inv,z)$~\cite{Schmid.1991,Derksen_Kraft.1997}. Noether's
theorem~\cite[p.~27]{Sturmfels.AIT} yields: $\beta(\Inv)\le |\G|=n!$,
which is not very informative. A much better bound exists for
permutation groups: \mbox{$\beta(\Inv)\le \binom{m}{2}=\binom{\binom n
    2} 2$}~\cite{Garsia_Stanton.1984}. However, our computations for
small $n$ shows that this is still a rather loose bound. This section
introduces the tools that produce this bound, and possibly even better
bounds.

A set of $m$ homogeneous invariants $(\theta_1,\dots,\theta_m)$ of
$I(G)$ is called a \emph{homogeneous system of parameters} or, for
short, a \emph{system of parameters} if the invariant ring $I(G)$ is
finitely generated over its subring $\K[\theta_1,\dots,\theta_m]$.
That is, if there exist a finite number of invariants
$(\eta_1,\dots,\eta_t)$ such that the invariant ring is the sum of the
subspaces $\eta_i.\K[\theta_1,\dots,\theta_m]$. By Noether's
normalization lemma, there always exists a system of parameters for
$I(G)$. Moreover, $I(G)$ is \emph{Cohen-Macaulay}, which means that
$I(G)$ is a free-module over any system of parameters. So, if the set
$(\eta_1,\dots,\eta_t)$ is minimal for inclusion, $I(G)$ decomposes
into a direct sum:
\begin{displaymath}
  I(G)=\bigoplus_{i=1}^t \eta_i . \K[\theta_1,\dots,\theta_m].
\end{displaymath}
This decomposition is called a \emph{Hironaka decomposition} of the
invariant ring. The $\theta_i$ are called \emph{primary invariants},
and the $\eta_i$ \emph{secondary invariants} (in algebraic
combinatorics literature, the $\theta_i$ are some times called
\emph{quasi-generators} and the $\eta_i$
\emph{separators}~\cite{Garsia_Stanton.1984}). It should be emphasized
that primary and secondary invariants are not uniquely determined, and
that being a primary or secondary invariant is not an intrinsic
property of an invariant $p$, but rather express the role of $p$ in a
particular generating set.

The primary and secondary invariants together form a generating
set. From the degrees $(d_1,\dots,d_m)$ of the primary invariants
$(\theta_1,\dots,\theta_m)$ and the Hilbert series we can compute the
number $t$ and the degrees $(e_1,\dots,e_t)$ of the secondary
invariants $(\eta_1,\dots,\eta_t)$ by the formula:
\begin{equation}
  \label{eq:degres_secondaires}
  z^{e_1}+\dots+z^{e_t}=(1-z^{d_1})\cdots(1-z^{d_m}) H(I(G),z).
\end{equation}
Assuming $d_1\le\dots\le d_m$ and $e_1\le\dots\le e_t$, it can be
proved that:
\begin{equation}
  \label{eq.max_secondaires}
  \begin{gathered}
    t   = \frac{d_1\cdots d_m}{|G|},\\
    e_t = d_1+\dots+d_m - m - \mu,\\
    \beta(I(G))\le \max(d_m,e_t),
  \end{gathered}
\end{equation}
where $\mu$ is the smallest degree of a polynomial $p$ such that
$\sigma\cdot p=\det(\sigma) p$ for all $\sigma\in G$ \cite[Proposition
3.8]{Stanley.1979}.

For example, if $G$ is the symmetric group $\sym_m$, the $m$
elementary symmetric polynomial (or the $m$ first symmetric power sums)
form a system of parameters, $t=1$, $e_t=0$ and $\eta_1=1$. This is
consistent with the fundamental theorem of symmetric polynomials.

More generally, for $\G$ as well as for any permutation group, the
elementary symmetric polynomials still form a system of parameters.
This yields the following information on $\Inv$.
\begin{proposition}
  \label{prop.bounds_symmetric}
  For $n\ge4$, consider the system of parameters of $\Inv$ composed
  of the elementary symmetric polynomials, and let $(e_1,\dots,e_t)$
  be the degrees of secondary invariants. Then,
  \begin{gather*}
    t   = \frac{m!}{|\G|} = \frac{\binom{n}2!}{n!},\\
    e_t = \binom m 2 - \mu_n = \binom{\binom{n}2}2- \mu_n,\\
    \beta(\Inv)\le \binom{\binom n 2} 2-\mu_n,
  \end{gather*}
  where $\mu_n=0$ if $n$ is even, and $\mu_n=\lceil\frac34(n-1)\rceil$
  otherwise.
\end{proposition}
For example, $\beta(\Inv[4])\le15$, $\beta(\Inv[5])\le42$ and
$\beta(\Inv[6])\le104$.
\begin{proof}
  We only have to check the value of $\mu_n$.  When $n$ is even,
  $\det(\sigma)=1$ for all $\sigma\in\G$. Therefore, $p:=1$ verifies the
  condition $\sigma\cdot p=\det(\sigma) p$ for all $\sigma\in\G$. We
  note that this is generally true for any Gorenstein ring (see
  \S~\ref{Gorenstein} and \cite[\S~8]{Stanley.1979}).  When $n$ is
  odd, the sign of a permutation of the edges is the sign of the
  corresponding permutation of the vertices. Then, the smallest degree
  $\mu_n$ of a polynomial $p$ such that $\sigma\cdot p=\sign\sigma p$
  is the smallest number of edges of multigraph with no odd
  automorphism. The following lemma completes the proof.
\end{proof}

\begin{lemma}
  The smallest number of edges of a multigraph $\g_n$ on $n\ge4$
  vertices without odd automorphism is $\lceil\frac34(n-1)\rceil$.
\end{lemma}
\begin{proof}
  Such multigraphs can be constructed for any $n$ as follows:
  \begin{itemize}
  \item $\g_4:=\graph{autogen/ugraph-101100}$;
    $\g_5:=\graph{autogen/ugraph-1011000000}$;
    $\g_6:=\graph{autogen/ugraph-101000000010001}$;
    $\g_7:=\graph{autogen/ugraph-101100000010001000000}$;
  \item $\g_{4k}$ is composed of $k$ copies of $\g_4$ ($3k$ edges);
  \item $\g_{4k+1}$ is composed of $k$ copies of $\g_4$ and an isolated
    vertex ($3k$ edges);
  \item $\g_{4k+2}$ is composed of $k-1$ copies of $\g_4$ and one copy of
    $\g_6$ ($3k+1$ edges);
  \item $\g_{4k+3}$ is composed of $k-1$ copies of $\g_4$ and one copy
    of $\g_7$ ($3k+2$ edges).
  \end{itemize}
  The minimality of the number of edges of such multigraphs can be
  proved by induction over $n$.
\end{proof}

So, the knowledge of a system of parameters and of the Hilbert series
provides both an upper bound on $\beta(\Inv)$, as well as bounds on
the coefficients of $\MGSGS(\Inv,z)$. Unfortunately, our experience
has shown that generating sets composed of primary and secondary
invariants are far from minimal (see Figures~\ref{fig.secondaries.low}
and~\ref{fig.gen}), so those bounds are quite loose. Moreover, to our
knowledge, those bounds are the only obtainable information about a
minimal generating set, without actually computing it.

\section{Low degrees systems of parameters}

\label{parameters}

We now search for a low degrees system of parameters for $\Inv$, in
order to improve the bound on $\beta(\Inv)$.
Equation~(\ref{eq:degres_secondaires}) can guide our quest by suggesting
possible degrees. Indeed, there can only exist a system of parameters
of degrees $(d_1,\dots,d_m)$ if the expression
$(1-z^{d_1})\cdots(1-z^{d_m}) H(I(G),z)$ is a polynomial with positive
integer coefficients. It has even been conjectured by Mallows and
Sloane~\cite{Mallows_Sloane.1973,Dixmier.1991} that the converse is
true: if $(1-z^{d_1})\cdots(1-z^{d_m}) H(I(G),z)$ is a polynomial with
positive integer coefficient, then there exists a system of parameters
of degrees $(d_1,\dots,d_m)$. A counter-example has been found, but
the conjecture still holds if the representation of $G$ over $V$ is
irreducible, or when using a multigraded Hilbert series (one grading
for each irreducible component)~\cite[p.~5]{Dixmier.1991}.

By tweaking the Hilbert series for $n\le21$, and the multigraded
Hilbert series for $n\le15$, we find that the degree sequence
$(1,\dots,n,\ 2,\dots,\binom{n-1}{2})$ always produces a polynomial
with positive integer coefficient. We also proved that, for any $n$,
this degree sequence produces a polynomial. We are therefore somewhat
confident with the following conjecture:
\begin{conjecture}
  \label{conj.degrees_parameters}
  For all $n\ge3$, there exists a system of parameters for $\Inv$ of
  degrees $(1,\dots,n,\ 2,\dots,\binom{n-1}{2})$. As a direct
  consequence, $\beta(\Inv)\le \binom n
  2+\binom{\binom{n-1}{2}}{2}-\mu_n$.
\end{conjecture}
For example, $\beta(\Inv[4])\le9$, $\beta(\Inv[5])\le22$ and
$\beta(\Inv[6])\le60$, which are much smaller degree bounds than those
provided by proposition~\ref{prop.bounds_symmetric}.
Figure~\ref{fig.secondaries} displays the number of secondary
invariants depending on the system of parameters.
\begin{figure}
  \begin{center}
    \subfigure[System of parameters: symmetric power sums]{
      \includegraphics[scale=0.5]{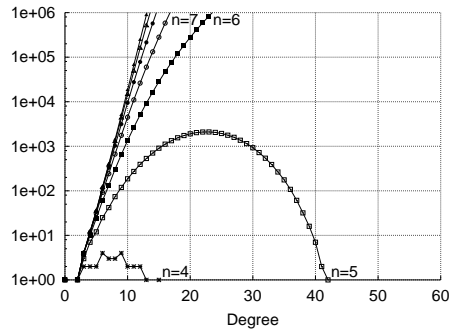}}
    \label{fig.secondaries.high}
    \subfigure[Conjectured system of parameters]{
      \includegraphics[scale=0.5]{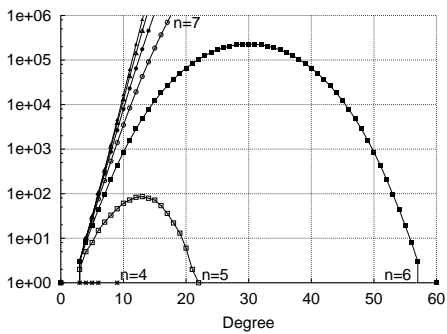}}
    \label{fig.secondaries.low}
  \end{center}
  \caption{Number of secondaries per degree}
  \label{fig.secondaries}
\end{figure}

Next, we construct a reasonable system of parameters and check its
validity for $n=3,4,5$, which proves
conjecture~\ref{conj.degrees_parameters} for those values. We note
that Dixmier~\cite{Dixmier.1991} constructed a system of parameters
with degrees $(2,3,4,5,6)$ for the representation $[3,2]$ of $\sym_5$,
and proved its validity by hand. By using the decomposition of the
representation $\G[5]$ into $[5]+[4,1]+[3,2]$, this also provides a
system of parameters with the expected degrees.

The form of the degree sequence suggests starting from the $\binom n
2$ first symmetric power sums, and replacing the last $n-1$ degrees
($\binom{n-1}{2}+1,\dots \binom n 2$) by some invariants of degree
$2,\dots,n$. Moreover, since the representation splits into
$[n]\oplus[n-1,1]$ and $[n-2,2]$ (see \S~\ref{valuated_graphs}), we
get a system of parameters for the invariant ring of the whole
representation, by taking systems of parameters for the invariant
rings of each components and putting them together. Recall that the
first component is the natural representation of $\sym_n$ by
permutations of the stars $(\E_1,\dots,\E_n)$. So, the invariant ring
over this component is the ring of symmetric polynomials in the dual
variables $(X_1,\dots,X_n)$, and the $n$ first symmetric power sums in
the $X_i$ form a system of parameters for this component. Note that,
up to a constant $2$, the first symmetric power sum in the $X_i$ is
equal to the first symmetric power sum in the $x_\ij$. All this leads
to the following conjecture:
\begin{conjecture}
  \label{conj.primaries}
  If $n\ge3$, the following system of invariants is a system of
  parameters for $\Inv$.
  \begin{gather*}
    x_{\{1,2\}}+\dots+x_{\{n-1,n\}}, \quad \dots, \quad
    x_{\{1,2\}}^{\binom{n-1}{2}}+\dots+x_{\{n-1,n\}}^{\binom{n-1}{2}},\\
    X_1+\dots+X_n,\quad\dots,\quad X_1^n+\dots+X_n^n.
  \end{gather*}
\end{conjecture}
\begin{proposition}
  Conjecture~\ref{conj.primaries} holds for $3\le n \le5$.
\end{proposition}
This is immediate for $n=3$, since the component \mbox{$[n-2,2]$} is
trivial. To test the conjecture for other small cases, we used the
following general characterization:
\begin{caracterisation}[\cite{Sturmfels.AIT}]
  \label{carac.parameters}
  A set of $m$ homogeneous invariants $(\theta_1,\dots,\theta_m)$ is a
  system of parameters if and only if $\v=0$ is the only common zero
  of the $\theta_i$:
  \begin{displaymath}
    \theta_1(\v)=\dots=\theta_m(\v)=0 \Rightarrow \v=0
  \end{displaymath}
\end{caracterisation}
For $n=4$, this characterization is enough to prove
conjecture~\ref{conj.primaries} by hand.

For $n\ge5$, we can try to check conjecture~\ref{conj.primaries} as
follow: compute a Gröbner basis for the $\theta_i$, and verify that
for each variable $x_\ij$ there is a polynomial in the Gröbner basis
whose leading term is of the form $x_\ij^k$
(see~\cite[Subroutine~2.5.2]{Sturmfels.AIT}; this is both a necessary
and sufficient condition for the radical of the ideal generated by the
$\theta_i$ to be the irrelevant ideal).

For $n=3,4$, the direct computation of this Gröbner basis takes less
than one second, but for $n=5$ it seems to be intractable and fails.
However, some equivalent Gröbner basis can be computed in about one
minute, by using a suitable linear change of basis which respects the
decomposition of the representation. The verification of
characterization~\ref{carac.parameters} is then straightforward.

For $n=6$, even with the same linear change of basis and using
\fgb~\cite{Faugere.1999}, the computation is intractable.  In fact,
the growth of the Gröbner basis seems to follow nearly the worst
possible theoretical case~\cite{Garsia_Wallach.1999}.

Finally, as a by-product of the computation of minimal generating sets
with \permuvar, we checked that for $n\le8$ and for the low-degree
homogeneous components, the ring of invariants is indeed a free-module
over $\K[\theta_1,\dots,\theta_m]$. This gives some confidence in
conjecture~\ref{conj.primaries}, but so far we are unable to prove it.


If the above construction of a system of parameters is correct, we
believe it to be nearly optimal: there exist no general construction
of systems of parameters with lower degrees. Indeed, we calculated,
for small values of $n$, the smallest degree sequence allowed by the
multigraded Hilbert series. For $3,4,5$, the degrees conjectured above
are optimal. For $n\ge6$, it was usually possible to divide some of
the degrees by $2$ or $3$. For example, for $n=6$ and $7$ the best
degree sequences are respectively $(1,\dots,6,\ 
2,\dots,7,\frac82,\frac93,\frac{10}2)$ and $(1,\dots,7,\ 
2,\dots,7,\frac82,9,\dots,13,\frac{14}2,15)$. We have not noticed any
regularity in these case by case optimizations. Therefore, we don't
think this technique can be refined much further in order to get
better degree bounds.

\section{Computing minimal generating sets}

\label{MGS}

Since we can not get more \apriori information on homogeneous minimal
generating sets of $\Inv$, we proceed with explicitly computing them.
The computations are very intensive (even for $n=5$) but give some
feeling of the size and degree of minimal generating sets. The bounds
given by Hironaka decompositions seem very loose.

The basic principle of the classical algorithms is to construct
generating sets degree by degree, from $1$ up to the best degree bound
known. Since the complexity of the computations involved usually
increases quickly with increasing degree, the quality of the degree
bound is crucial. One can take advantage of the existence of a
Hironaka decomposition by computing secondary invariants and, while
doing so, selecting the secondary invariants that are
\emph{irreducible} (\ie that cannot be expressed as products of lower
degree secondary invariants). The irreducible secondary invariants
together with the primary invariants form a minimal generating set
(some primary invariants may need to be removed).

Most software~\cite{Kemper.Invar} relies on a precomputation of a
Gröbner basis of the system of parameters to greatly speed up the
rest of the computations. However, with $\Inv$ this precomputation is
very hard, if not impossible (see \S~\ref{parameters}). Software that
do not rely on this precomputation uses linear algebra on the
homogeneous component of degree $d$ of the whole ring of polynomials,
whose dimension grows quickly with increasing $d$, and fail early.

Since our group is a permutation group, invariants can be stored as
linear combinations of invariants $\exps\g$. This saves a lot of
memory (up to a factor of $1/|\G|$ for monomials without symmetries,
which happens to be the case for most of them). This data structure
also allows for the same linear algebra operations inside the
homogeneous component of degree $d$ of the invariant ring which is
considerably smaller. We therefore implemented our own invariant
theory software \permuvar which takes advantage of the particular
properties of permutation groups~\cite{Thiery.P.2001}. We chose the
computer algebra system \mupad, which is freely available (but alas
not open source software), and allows modularity through object
oriented programming. Moreover, \mupad's dynamic modules will allow
for rewriting critical sections in a very efficient language like
\texttt{C++}.

A sketch of the algorithm follows. We denote by
$\langle\theta_1,\dots,\theta_m\rangle_d$ the homogeneous component of
degree $d$ of the ideal generated by $(\theta_1,\dots,\theta_m)$ in $I(G)$.
\begin{algorithm}[Computing secondary invariants]
  \label{algo.sec}
  \textbf{Input}: a system of parameters $(\theta_1,\dots,\theta_m)$
  and a function $\texttt{nextInvariant}(d)$ which iterates through a
  set of invariants of degree $d$ spanning $I(G)_d$ as a
  vector space.\\
  \textbf{Output}: (irreducible) secondary invariants.

\begin{algorithmic}
  \newcommand{\nextInvariant}{\texttt{nextInvariant}\xspace}
  \renewcommand{\d}{\texttt{d}}
  \newcommand{\p}{\texttt{p}\xspace}
  \renewcommand{\L}{\texttt{L}\xspace}
  \renewcommand{\sec}{\texttt{secondaries}\xspace}
  \newcommand{\irred}{\texttt{irreducibles}\xspace}
\FOR{\d\ from $1$ to $e_t$} {
  \STATE{// Compute secondary invariants for degree \d}
  \STATE{\sec[\d]:=[\ ]; // Secondary invariants}
  \STATE{\irred[\d]:=[\ ];  // Irreducible secondaries}
  \STATE{// Compute a basis $\L$ of $\langle\theta_1,\dots,\theta_m\rangle_d$}
  \STATE{\L:=[\ ];}
  \FOR{\p product of a previous secondary and a
    non-trivial product of the $\theta_i$}{
    \STATE{insert \p into \L;}
    }
  \ENDFOR

  \STATE{// Extend $\L$ to a basis of $\K[I(G)_{<d}]_d$}

  \FOR{\p product of previous secondaries} {
    \IF{\p is not in the vector space spanned by \L} {
      \STATE{insert \p into \sec[\d];}
      \STATE{insert \p into \L;}
      }
    \ENDIF
    }
  \ENDFOR

  \STATE{// Construct the irreducible secondaries}
  \WHILE{\p:=\nextInvariant(d)} {
    \IF{\p is not in the vector space spanned by \L} {
      \STATE{insert \p into \sec[\d]}
      \STATE{insert \p into \irred[\d]}
      \STATE{insert \p into \L}
      }
    \ENDIF
    }
  \ENDWHILE
  }
\ENDFOR
\end{algorithmic}
\end{algorithm}
Some comments about this algorithm are in order:

(i) In the last loop, the Hilbert series provides a stopping
condition, since the number of secondary invariants is known. To
maintain efficiency, the elements of $L$ are mutually reduced by Gauss
elimination. Testing if $p$ is in the vector space generated by $L$
amounts to reducing it modulo $L$; inserting it into $L$ amounts to
further reducing the elements of $L$ by $p$. Therefore, this algorithm
is essentially a step by step matrix inversion by Gauss elimination,
and the cost for each degree is about $(\dim I(G)_d^3$.

(ii) The main waste of memory and time in this algorithm is the
explicit computation of the vector space basis $\texttt{L}$ of
$\langle\theta_1,\dots,\theta_m\rangle_d$. It would be nice to work
directly in the quotient of $I(G)$ by the ideal
$\langle\theta_1,\dots,\theta_m\rangle$, as in the algorithm based on
a Gröbner basis precomputation. This approach is further developed
in~\cite{Thiery.CMGS.2001}.

(iii) By properly keeping track of the reductions in \texttt{L}, we
can determine which primary invariants should be removed in order to
obtain a minimal generating set. In addition, by properly choosing the
$\texttt{nextInvariant}$ function, we can check whether a given set of
invariants is generating, and if not construct counter-examples.

For $n=5$, we could only compute a partial minimal generating set
$S_5$ up to degree $10$, whereas the best a priori degree bound is
$\beta(\Inv)\le22$. However, $s_{10}(\Inv)=0$ (\ie a minimal
generating set contains no invariant of degree $10$), and
Figure~\ref{fig.gen} strongly suggested that $s_d(\Inv)=0$ for any
$d\ge10$. This has been checked by Kemper~\cite{Kemper.IG5}, using
\emph{ad hoc} computations, thus proving that $S_5$ is a minimal
generating set (see \S~\ref{unimodality} for a possible alternative
approach). Therefore, $\beta(\Inv[4])=5$, and $\beta(\Inv[5])=9$.
\begin{conjecture}
  If $n\ge4$, $\beta(\Inv)=\binom n 2 -1$.
\end{conjecture}

\begin{figure}
  \begin{center}
    \includegraphics[scale=\gnuplotscale]{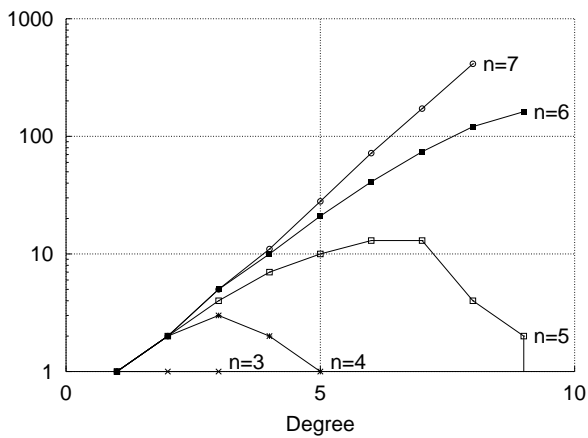}
  \end{center}
  \caption{$s_d(\Inv)$: number of invariants per degree $d$ in a minimal
    generating set of $\Inv$}
  \label{fig.gen}
\end{figure}

\section{The Gorenstein property}

\label{Gorenstein}

In this section, we show that the invariant ring $\Inv$ is Gorenstein
when $n$ is even, which indicates several duality properties of
$\Inv$. In particular, $e_t = d_1+\dots+d_m - m$ and there are as many
secondary invariants of degree $e_t-d$ and degree $d$. Actually, this
could be used to considerably speed up the construction of secondary
invariants~\cite{Thiery.P.2001}.

\begin{lemma}
  (i) Let $\sigma$ be a permutation of the vertices, that is an
  element of $\sym_n$, and $\overline\sigma$ be the corresponding
  permutation of the edges in $\G$. Then:
  \begin{xalignat*}{2}
    \sign(\overline\sigma)&=\sign(\sigma) && \text{if $n$ is odd},\\
    \sign(\overline\sigma)&=1 && \text{if $n$ is even}.
  \end{xalignat*}

  (ii) If $n$ is even, $\G$ is a subgroup of the special linear group
  $\SL(V)$.

  (iii) If $n$ is odd, the representation of $\sym_n$ on the
  irreducible component $[n-2,2]$ is a subgroup of $\SL(V)$.
\end{lemma}
\begin{proof}
  (i) If $\sigma$ is a transposition of $2$ vertices, then
  $\overline\sigma$ exchanges $n-2$ pairs of edges, and
  $\sign(\overline\sigma)=(-1)^{n-2}$.

  (iii) Take $\sigma\in\sym_n$, and $M$ the matrix of the
  representation of $\sigma$ on the component $[n-2,2]$. The
  determinant of the representation of $\sigma$ on $\V$ is
  $\sign(\overline\sigma)$, whereas the sign of the representation of
  $\sigma$ on the other component $[n]\oplus[n-1,1]$ is
  $\sign(\sigma)$ (natural representation of $\sym_n$). Therefore,
  $\det(M)=\sign(\overline\sigma)/\sign(\sigma)=1$, if $n$ is odd.
\end{proof}

Then, Watanabee's theorem~\cite[\S~8]{Stanley.1979} applies.
\begin{theorem}
  (i) When $n$ is even, $\Inv$ is Gorenstein.

  (ii) When $n$ is odd, the invariant ring over the irreducible
  component $[n-2,2]$ is Gorenstein.
\end{theorem}

\section{The chain product}

\label{chain_product}

We have discussed the power of the grading of the invariant ring. We
now define another product on the invariant ring $\Inv$, called the
\emph{chain product}, which preserves a finer grading and has a nice
computational behavior. Most algebraic properties of the invariant
ring with respect to the chain product transfer back to the usual
product. We only construct and use the chain product for $\Inv$, but
it generalizes to any permutation group~\cite{Thiery.P.2001}.

Let $\g$ be a multigraph. As in the following example, it can be
interpreted as a superposition of simple graphs
$\g_1,\g_2,\cdots,\g_k$, where
$\g_1\supseteq\g_2\supseteq\cdots\supseteq\g_k$:
\begin{displaymath}
  \graph{autogen/graph-1300003031}
  \quad \longleftrightarrow \quad
  \graph{autogen/graph-1100001011} \supseteq
  \graph{autogen/graph-0100001010} \supseteq
  \graph{autogen/graph-0100001010}
\end{displaymath}
Thus, $\g$ can be identified with the \emph{multichain} (\ie chain
with repetitions)
$C(\g):=\g_1\supseteq\g_2\supseteq\cdots\supseteq\g_k$ of simple
graphs. The \emph{shape} $\lambda(\g)$ of $\g$ is the decreasing
sequence of the sizes of the simple graphs in $C(\g)$. Here,
$\lambda(\g)=(5,3,3)$. A polynomial is called
\emph{finely-homogeneous} if all its monomials have the same shape.
Since two monomials $\expm\g$ and $\expm{\g'}$ in the same
$\sym_n$-orbit have the same shape, any invariant decomposes into a
sum of finely-homogeneous invariants. Therefore, the shape defines a
fine grading on the invariant ring $\Inv$, and we denote by
$\Inv[n,(5,3,3)]$ the \emph{finely homogeneous component} of $\Inv$
for the shape $(5,3,3)$.

The usual product does not preserve this grading:
\begin{displaymath}
  \pgraph{autogen/graph-100000}\pgraph{autogen/graph-100000}=
  \pgraph{autogen/graph-200000}+
  2\pgraph{autogen/graph-110000}+
  2\pgraph{autogen/graph-100001}.
\end{displaymath}
The \emph{chain product} $\expm\g\star\expm\h$ of two monomials
$\expm\g$ and $\expm\h$ is the usual product of $\expm\g$ and
$\expm\h$ if the two multichains $C(\g)$ and $C(\h)$ can be merged
into another multichain, and zero otherwise. The chain product extends
to invariants, and yields for example:
\begin{displaymath}
  \pgraph{autogen/graph-100000}\star\pgraph{autogen/graph-100000}=
  \pgraph{autogen/graph-200000}
\end{displaymath}
The chain product preserves the fine grading of the invariant ring,
since the shape $\expm\g\star\expm\h$ can be obtained by merging the
shapes of $\g$ and $\h$.

The invariant ring $\Inv$ together with the chain product is actually
isomorphic to the Stanley-Reisner ring of the poset of unlabelled
graphs on $n$ vertices ordered by subgraph. Stanley-Reisner rings of
posets have been intensively studied, in particular by Garsia and
Stanton~\cite{Garsia_Stanton.1984} to construct Hironaka
decompositions of invariant rings of certain permutation groups. We
did not succeed in using this theoretical framework to get a Hironaka
decomposition of $\Inv$. The need of taking the elementary symmetric
polynomials as a system of parameters causes the main difficulty.
Indeed, this is not a low degree system of parameters and there are
too many secondary invariants. However, even our naive point of view
of the Stanley-Reisner ring as an alternative product on $\Inv$ yields
dramatic speed ups of the computations.

The following proposition is the heart of this technique.
\begin{proposition}[\cite{Garsia_Stanton.1984}]
  A Hironaka decomposition of $\Inv$ for the chain product, is also a
  Hironaka decomposition of $\Inv$ for the usual product.
\end{proposition}
The key of the proof is that, if $p$ and $q$ are finely homogeneous,
the maximal finely homogeneous component of $pq$ is exactly $p\star
q$. The result follows by induction over the fine grading. We used the
same principle to prove a similar result on generating sets.
\begin{proposition}
  \label{prop.transfert_generating}
  A generating set of $\Inv$ for the chain product is a generating set
  of $\Inv$ for the usual product.
\end{proposition}
In all our examples, however, minimal generating sets for the chain
product were far from being minimal for the usual product.

The elementary symmetric polynomials form a system of parameters for
the chain product. We do not know if there are other systems of
parameters, since the usual characterization from
proposition~\ref{carac.parameters} does not apply for the chain
product. In particular, the symmetric power sums do not form a system
of parameters for the chain product. They are not even algebraically
independent since $\sum x_\ij^k=(\sum x_\ij)^k$. Given the size of the
minimal generating sets we computed, there are no systems of
parameters for the chain product with degrees as low as in
conjecture~\ref{conj.degrees_parameters}.

In~\cite{Thiery.P.2001}, we describe how to use this product for
faster computations. Practically, we could push the computation of a
partial minimal generating set $S$ for $\Inv[5]$ up to the degree $22$
instead of only $10$. This is a significant progress, considering that
the dimension of $\Inv[5,22]$ is $174403$, whereas the dimension of
$\Inv[5,10]$ is only $974$. Unfortunately, we cannot use a low-degrees
system of parameters, so the degree bound is $42$ instead of $22$.
This means that there is still a lot of work to do to get a full
minimal generating set for the chain product. On the other hand, this
partial computation yields a generating set for the usual product,
since $\beta(\Inv[5])\le22$.
\begin{proposition}
  The computed set $S$ is a generating set of $\Inv[5]$ for the usual
  product. However, $S$ has more than one thousand invariants of
  degree up to $22$.
\end{proposition}

To conclude, the usual product allowed us to compute a small set, with
is minimal, but not necessarily generating, whereas the chain product
allowed us to compute a set which is generating, but far from being
minimal.

\section{The invariant ring for $n=\infty$}

\label{infinity}

In this section, we study the projective limit $\Inf$ of the invariant
ring, and get back some information on $\Inv$.

A multigraph $\g$ on $n'\le n$ non-isolated vertices can be identified
with a multigraph on $n$ vertices by adding $n-n'$ isolated vertices.
This defines $\exps\g$ in $\Inv$. The set $B_n$ of all invariants
$\exps\g$, where $\g$ is a multigraph on less than $n$ non-isolated
vertices, is obviously a vector space basis of $\Inv$. For $n'\le n$,
let $\Phi_{n'}$ be the linear projection from $\Inv$ to $\Inv[n']$
which maps $\exps\g$ (in $\Inv$) to $0$ if $\g$ has strictly more than
$n'$ non-isolated vertices, and to $\exps\g$ (in $\Inv[n']$)
otherwise. Our definition of the exponential (see
\S~\ref{invariants_permutation}) makes it a surjective morphism of
graded algebra.  The projective limit of $\Inv$:
\begin{displaymath}
  \Inv[1] \stackrel{\Phi_1}\twoheadleftarrow
  \Inv[2] \stackrel{\Phi_2}\twoheadleftarrow
  \dots   \twoheadleftarrow
  \Inv[n] \stackrel{\Phi_n}\twoheadleftarrow
  \dots                    \twoheadleftarrow
  \Inf,
\end{displaymath}
defines a graded algebra $\Inf$, with a canonical vector space basis
$B_\infty:=\{\exps\g\}$ indexed by the multigraphs $\g$ on a finite
number of non-isolated vertices.

\begin{proposition}
  (i) $\Inf$ is the free polynomial ring over $C:=\{\exps\g \st \text{
    $\g$ is connected}\}$.

  (ii) The canonical morphism of graded algebra $\Phi_n:
  \Inf\twoheadrightarrow\Inv$ is an isomorphism up to the degree
  $\lfloor\frac n2\rfloor$.
\end{proposition}
\begin{proof}
  (i) Following the proof of proposition~\ref{prop.connected_generate}
  (ii), $C$ generates $\Inf$. Now, let $g_1,\dots,g_k$ be $k>0$
  connected multigraphs.  In the product
  $\exps{\g_1}\cdots\exps{\g_k}$, there is a term $\exps\h$ with
  coefficient $1$, where $\h$ is the disconnected multigraph whose
  connected components are precisely the $g_i$. This term is a marker
  of the product $\exps{\g_1}\cdots\exps{\g_k}$ in any non-trivial
  polynomial combination of elements of $S$. The algebraic
  independence follows.

  Any multigraph with $d$ edges and no isolated vertices has
  less than $2d$ vertices. (ii) follows.
\end{proof}

\begin{corollary}
  $\beta(\Inv)\ge \lfloor\frac n2\rfloor$.
\end{corollary}
This lower bound is loose: for $n\le5$, we know that
$\beta(\Inv)\ge\binom n 2-1$ and for $11\le n\le18$, it follows from
theorem~\ref{th.not_gen} (ii) that $\beta(\Inv)\ge n-2$. We expect
that refining this technique will yield much better lower bounds.

By (ii), the Hilbert series $H(\Inf,z)$ is the limit of the Hilbert
series $H(\Inv,z)$ as $n$ goes to infinity, and by (i)
\begin{displaymath}
  H(\Inf,z)=\prod_{d=1}^\infty \frac{1}{(1-z^d)^{n_d}},
\end{displaymath}
where $n_d$ is the number of connected multigraphs with $d$ edges. We
do not know how to directly compute $H(\Inv,z)$, or whether there exists a
closed form formula. The only asymptotic studies we have seen in the
literature deal with $n$ fixed and $d$ going to
infinity~\cite{Harary_Palmer.GE}.

\section{Unimodality}

\label{unimodality}

A startling fact revealed by our computations of minimal generating
sets (MGS) lies in Figure~\ref{fig.gen}, which shows the coefficients
of $\MGSGS(\Inv,z)$. For $n\le4$ and most likely for $n=5$, this
polynomial is \emph{unimodal}: the coefficients first increase
with the degree, and then decrease down to $0$.
\begin{conjecture}
  The polynomial $\MGSGS(\Inv,z)$ is unimodal.
\end{conjecture}
This would prove that the partial minimal generating set we computed
for $n=5$ is generating, and provide a very nice stopping condition
for algorithm~\ref{algo.sec}.

To figure out which properties of $\G$ could be useful to prove this
conjecture, we extend it to general groups of matrices. A finite
subgroup $G$ of $\GL(V)$ is \emph{MGS-unimodal} if the polynomial
$\MGSGS(I(G),z)$ is unimodal.
\begin{problem}
  Characterize MGS-unimodal groups.
\end{problem}
Not all groups are MGS-unimodal. Indeed, let $G$ be the subgroup of
$\GL(\C^2)$ generated by the matrix
\begin{displaymath}
  M:=
  \begin{bmatrix}
    j & 0\\
    0 & \overline j\\
  \end{bmatrix},
\end{displaymath}
where $j$ and $\overline j$ are the two non-trivial third roots of
unity. Obviously, $(x_1x_2, x_1^3, x_2^3)$ is a minimal generating set
of $I(G)$, and $\MGSGS(G,z)=z+0z^2+2z^3$, which is not unimodal.

Whereas the irreducible representations of the symmetric group are
thoroughly described, their invariants rings are barely known. In an
amazing but very technical paper~\cite{Dixmier.1991}, Dixmier has been
able to construct by hand minimal generating sets for several
irreducible representations of $\sym_n$, including all the irreducible
representations of $\sym_1,\dots,\sym_5$, except $[3,1,1]$. It follows
that the representations $[n]$,$[n-1,1]$, $[2,2]$ and $[3,2]$ are
MGS-unimodal, whereas the representations $[2,1^{n-2}]$ for $n\ge4$
and $[2,2,1]$ are not. This proves the existence non-MGS-unimodal
irreducible representations of the symmetric group.

The trivial group, the full symmetric group and multisymmetric
polynomials are MGS-unimodal. We checked with \permuvar that several
other small permutation groups are MGS-unimodal. It's tempting to
conjecture that all permutation groups are MGS-unimodal, since they
give rise to a lot of unimodality properties (see \cite{Stanley.1989};
note that, as opposite to here, the corresponding series are always
either log-concave or symmetric). However, for $n\ge4$, the
alternating group $\mathcal{A}_n$ is not MGS-unimodal. Indeed,
$I(\mathcal{A}_n)$ is generated by the elementary symmetric
polynomials of degrees $1,\dots,n$ together with the Van-der-Monde
determinant $\prod_{i<j} (x_i-x_j)$ of degree $\binom n 2$.

Figure~\ref{fig.secondaries} also shows that, up to $n=21$, the
generating series of the secondary invariants is unimodal (except at
$d=0$ and possibly $d=e_t$), and very smooth.
\begin{conjecture}
  Let $G$ be a permutation group, and $(\theta_1,\dots,\theta_m)$ be a
  system of parameters of $I(G)$. Then, the generating series of the
  secondary invariants is unimodal, except for $d=0$ and possibly
  $d=e_t$.
\end{conjecture}
We recall that this series can be computed directly from the Hilbert
series. Therefore a careful study of the Hilbert series might yield a
simple proof of this conjecture.

\section{The invariant ring over digraphs}

\label{digraphs}

In~\cite{Grigoriev.1979}, Grigoriev introduced a related invariant
ring, the \emph{invariant ring over digraphs} (digraphs are directed
graphs, with loops). The definition is similar to the one for the
invariant ring over graphs, but there are $n^2$ variables
$(x_{1,1},x_{1,2},\dots,x_{n,n})$, indexed by the pairs $(i,j)$ of
$\{1,\dots,n\}$. The action of $\sym_n$ is then defined by
$\sigma\cdot x_{i,j}:=x_{\sigma(i),\sigma(j)}$. In this section, we denote
by $\OInv$ the invariant ring over \emph{digraphs}. More generally,
Grigoriev defined the invariant ring over oriented $k$-hypergraphs,
with $n^k$ variables indexed by $k$-uples of $\{1,\dots,n\}$.

Lemma~1 of \cite{Grigoriev.1979} states that $\OInv$ is generated by
the invariants $\exps\g$, where $\g$ is a simple digraph. The proof is
said to be an easy generalization of the usual proof of the
fundamental theorem of symmetric functions. This surprised us, since
we proved this was false in $\Inv$ (theorem~\ref{th.not_gen} (ii)).
Therefore, we checked the condition~\ref{cond.dimension}, which failed
even for $n=3$ and degree $5$. We then ran \permuvar to try to compute
a minimal generating set using only simple digraphs. It failed as
expected, and produced the two following very small invariants, which
are not generated by simple digraphs:
\begin{xalignat*}{2}
  \pgraph{autogen/udigraph-020001000},&&
  \pgraph{autogen/udigraph-010002000}.
\end{xalignat*}
These counter-examples to lemma~1 of \cite{Grigoriev.1979} can also
easily be checked by hand.

This was disappointing. Indeed, the invariant ring $\Inv$ is the
quotient ring of $\OInv$ by the ideal generated by $x_{i,i}=0$ and
$x_{i,j}=x_{j,i}$. Therefore, we could have used lemma~1 to prove that
$\Inv$ is generated by the invariants where each variable appears with
degree at most $2$. This would provide a pretty good degree bound
$\beta(\Inv)\le 2\binom n 2$. Moreover, we could use this together
with computations of partial minimal generating sets to prove that
$\Inv[5]$ is generated by the invariants $\exps\g$, where $\g$ is a
multigraph with at least one isolated vertex, result of interest for
the reconstruction problem.

Most of the results on $\Inv$ apply as well for $\OInv$, but since the
number of variables is greater, the computations are even harder than
for $\Inv$, even if we ignore loops.

\section{The field of invariant fractions}

\label{invariant_field}

The \emph{field of invariants} $\K(x_\ij)^{\sym_n}$ is the subfield of
all rational fractions of $\K(x_\ij)$ which remain invariant under the
action of the group. The following classical lemma is valid for any
finite group of matrices.
\begin{lemma}
  \label{lem.fraction}
  The field of invariant is exactly the field of fractions of the
  invariant ring.
\end{lemma}
\begin{proof}
  By averaging over the group, write any invariant fraction as $\frac
  pq$ where $p$ and $q$ are invariant polynomials.
\end{proof}

In~\cite{Grigoriev.1979}, Grigoriev used basic Galois theory to prove
the existence of a generating set of the field of invariants
composed of $m+1$ invariants of degree less than $m$. The principle is
to first take the $m$ elementary symmetric polynomials, and to
consider the subfield of symmetric fractions. Since the ground field
$\K$ has characteristic zero (this would be also the case for any
normal ground field, like a finite field), the primitive element
theorem applies: there exist a primitive element $p$ which generates
the field of invariants over the field of symmetric fractions.
Therefore, the $m$ elementary symmetric polynomials together with $p$
generates the field of invariants.

However, Grigoriev did not provide a way to construct such an element.
Moreover, the proof that it could be chosen of degree less than $m$
was incorrect, since it relied on lemma~1 of \cite{Grigoriev.1979}
which we disproved in \S~\ref{digraphs}.

\begin{theorem}
  Let $n\ge5$.  The field of invariants over graphs (respectively over
  digraphs) is generated by the elementary symmetric polynomials
  together with:
  \begin{xalignat*}{2}
    p:=\pgraph{autogen/ugraph-1100000000}, &&
    \text{ respectively }
    p:=\pgraph{autogen/udigraph-0110000000000000000000000}.
  \end{xalignat*}
\end{theorem}
\begin{proof}
  Key fact: in both cases a permutation of the edges belongs to the
  group if and only if it leaves $p$ invariant.
\end{proof}

Grigoriev also stated that such a generating set would be a complete
system of invariants. This is incorrect since, unlike a generating set
of the invariant ring, a generating set of the field of fraction is
not necessarily a complete system of invariants. For example, our
generating sets do not separate the following pairs of non-isomorphic
graphs:
\begin{xalignat*}{2}
  \left\{
    \graph{autogen/ugraph-1110000000},
    \graph{autogen/ugraph-1101000000}
  \right\}; &&
  \left\{
    \graph{autogen/udigraph-0110000000000110000000000},
    \graph{autogen/udigraph-0110000000000000000000110}
  \right\}.
\end{xalignat*}

In some cases, the field of invariants can be used to indirectly apply
Galois theory on the invariant ring.
\begin{theorem}
  If $n\neq4,5,6,8$, there is no intermediate invariant ring of matrix
  group between the ring of symmetric polynomials (respectively the
  ring of alternate polynomials for $n$ even) and the ring of
  invariants $\Inv$.
\end{theorem}
\begin{proof}
  For $n\neq4,5,6,8$, the group $\G$ is a maximal proper subgroup of
  the symmetric group $\sym_m$ (respectively the alternate group
  $\mathcal A_m$ for $n$ even) \cite{Faradzev_al.IATCO}. Basic Galois
  theory then proves the theorem for the field of invariants, and
  lemma~\ref{lem.fraction} transfers it back to the invariant ring.
\end{proof}

\section{Conclusion}

Invariant theory provides both very general tools and algorithms to
study the invariant ring $\Inv$ over graphs. Unfortunately, the
computer exploration of small cases appears to be very hard and shows
that those tools and algorithms lack accuracy and efficiency for our
particular invariant ring. However, we could still obtain a few
results, formulate conjectures related to $\Inv$, and solve a problem
arising from graph theory.

\section*{Acknowledgments}

This research was partially funded by the Région Rhône-Alpes. We
gratefully thank A.~Garsia, J-C. Faugère and N.~Wallach for the time
invested trying to compute a Gröbner basis for
conjecture~\ref{conj.primaries}. In particular, discussions with
A.~Garsia where very helpful, and raised decisive ideas. Finally, we
would like to thank M. Pouzet for introducing and guiding us through
this beautiful subject.

\bibliographystyle{acm}
\bibliography{main}

\end{document}